\documentclass[reqno,12pt]{article}
\usepackage{a4wide,color,eucal,enumerate,mathrsfs,amsbsy}
\usepackage[normalem]{ulem}
\usepackage{amsmath,amssymb,epsfig,bbm,graphicx}
\numberwithin{equation}{section}
\usepackage{tikz}

\usepackage[latin1]{inputenc}
\usepackage{hyperref}
\usepackage{pdfsync}

\newcommand{\R}{\mathbb{R}}

\newcommand{\bb}{{\mbox{\boldmath$b$}}}

\newcommand{\mm}{{\mbox{\boldmath$m$}}}

\newcommand{\sbb}{{\mbox{\scriptsize\boldmath$b$}}}

\newcommand{\ppi}{{\mbox{\boldmath$\pi$}}}

\newcommand{\sfd}{{\sf d}}

\newcommand{\Kliminf}{K\kern-3pt-\kern-2pt\mathop{\rm lim\,inf}\limits}  % Kuratowski liminf di insiemi
\renewcommand{\div}{\mathop{\rm div}\nolimits}          %Divergence
\renewcommand{\d}{{\mathrm d}}

\newcommand{\restr}[1]{\lower3pt\hbox{$|_{#1}$}}

\newcommand{\up}{\uparrow}
  
\newcommand{\nchi}{{\raise.3ex\hbox{$\chi$}}}
\newcommand{\weakto}{\rightharpoonup}
\newenvironment{proof}{\removelastskip\par\medskip   % inizio e fine dimostrazione
\noindent{\em Proof.}
\rm}{\penalty-20\null\hfill$\square$\par\medbreak}

\newtheorem{theorem}{Theorem}[section]

\newtheorem{lemma}[theorem]{Lemma}
\newtheorem{proposition}[theorem]{Proposition}
\newtheorem{definition}[theorem]{Definition}

\newtheorem{remark}[theorem]{Remark}

\newcommand{\prob}[1]{\mathcal P(#1)}

\newcommand{\e}{{\rm{e}}} % mappa di valutazione, bisogna mettere `a mano' il tempo t

\renewcommand{\mm}{\mathfrak m}

\newcommand{\ep}{\varepsilon}

\newcommand{\Der}[2]{{ \rm Der}^{{#1},{#2}}}
\newcommand{\Derr}[1]{{ \rm Der}^{#1}}
\newcommand{\Derri}{{ \rm Der}_b}
\newcommand{\Deri}{{ \rm Der}_{L}}
\newcommand{ \Derii}{\Deri^{\infty}}

\newcommand{\sppi}{{\mbox{\scriptsize\boldmath$\pi$}}}

\newcommand{\Lipb}{{\rm Lip}_b(X,\sfd)}
\newcommand{\Lipk}{{\rm Lip}_0(X,\sfd)}

\newcommand{\J}I
\title{Sobolev and BV spaces on metric measure spaces \\ via derivations and integration by parts}

\begin{document}
\author{Simone Di Marino\thanks{Scuola Normale Superiore, Pisa. email: \textsf{simone.dimarino@sns.it}} }

\maketitle

\begin{abstract}
We develop a theory of BV and Sobolev Spaces via integration by parts formula in abstract metric spaces; the role of vector fields is played by Weaver's metric derivations. The definition hereby given is shown to be equivalent to many others present in literature.
\end{abstract}

\section*{Introduction}

In the last few years a great attention has been devoted to the
theory of Sobolev spaces $W^{1,q}$ on metric measure spaces
$(X,\sfd,\mm)$, see for instance \cite{Heinonen07,Hajlasz-Koskela,Bjorn-Bjorn11}
for an overview on this subject and \cite{AGS11} for more recent developments.
These definitions of Sobolev spaces
usually come with a weak definition of modulus of gradient, in
particular the notion of $q$-upper gradient has been introduced in
\cite{Koskela-MacManus} and used in \cite{Shanmugalingam00} for a
Sobolev space theory. Also, in \cite{Shanmugalingam00} the notion of
minimal $q$-upper gradient has been proved to be equivalent to the
notion of relaxed upper gradient arising in Cheeger's paper
\cite{Cheeger00}. In \cite{AGS11} the definitions of $q$-relaxed slope and $q$-weak upper gradient are given and the minimal ones are seen to be equivalent to the ones in \cite{Shanmugalingam00}.

All of those approach give us a notion of \emph{modulus of the gradient} instead of the gradient itself, and an integration by parts formula is present only in special cases, and moreover it is often only an inequality. In this paper we want to fill this gap, namely giving a definition of Sobolev spaces more similar to the classical one given with an integration by parts formula; in $\R^n$ this formula can be written as
$$ \int_{\R^n} {\rm div} (\textbf{v}) f \, \d x = -\int_{\R^n} \textbf{v} \cdot \nabla f \, \d x, $$
where $f$ and $\textbf{v}$ are smooth functions and $\textbf{v}$ is with compact support. The usual definition of the Sobolev space $W^{1,p}$ then can be seen like this: $f \in L^p $ is a Sobolev function if there exists a function $g_f \in L^p(\R^n ; \R^n)$ such that
\begin{equation}\label{eqn:weake1} \int_{\R^n} {\rm div}(\textbf{v}) f \, \d x = - \int_{\R^n} \textbf{v} \cdot g_f \, \d x  \qquad \forall \textbf{v} \in C^{\infty}_c( \R^n;\R^n ). \end{equation}
Another equivalent formulation is that there exists a constant $C$ such that
\begin{equation}\label{eqn:weake2} \left| \int_{\R^n} {\rm div}(\textbf{v}) f \, \d x \right| \leq C \cdot \| \textbf{v} \|_{L^q(\R^n;\R^n)} \qquad \forall \textbf{v} \in C^{\infty}_c( \R^n;\R^n ),\end{equation}
where one can recover the weak gradient $g_f$ by a simple duality argument when $p >1$. Note that \eqref{eqn:weake1} and \eqref{eqn:weake2} give the same space when $1<p< \infty $,  while they differ with $p=1$: in that case \eqref{eqn:weake1} is the definition for $W^{1,1}(\R^n)$ while \eqref{eqn:weake2} is the usual definition for the space $BV(\R^n)$. 
This definition can be generalized in any metric measure space; the problem is to find the correct generalization of vector fields in an abstract metric space are the \emph{derivations}.

The derivations were introduced in the seminal papers by Weaver, and then in more recent times widely used in the Lipschitz theory of metric spaces, for example in connection with Rademacher theory for metric spaces, but also as a generalization of sections of the tangent space \cite{Schioppa, Schioppacur, Bate, Gigli14, atrev}. Here we see that the derivations are also powerful tools in the Sobolev theory, as already point out in \cite{Gigli14}. A derivation is simply a linear map $\bb : \Lipk \to L^0(X,\mm)$ such that the Liebniz rule holds and it has the locality property $| \bb (f) | \leq g \cdot  {\rm lip}_a (f)$ for some $g \in L^0(X,\mm)$. Now we simply say that $f \in L^p$ is a function in $W^{1,p}$ if there is a linear map $L_f$ such that integration by part holds:
$$ \int_X L_f(\bb) \, \d \mm = - \int_X f \cdot \div \bb \, \d \mm \qquad \forall \bb \in \Der qq,$$ 
where $\Der qq$ is the subset of derivation for which $|\bb|, \div \bb \in L^q(X,\mm)$.

We will see that it is well defined a proper ``differential" $df: \Der qq \to L^1$, and so it is possible to provide also a notion of modulus of the gradient $|\nabla f|$ in such a way that $| df(\bb) | \leq | \nabla f| \cdot |\bb|$; in Section \ref{sec:eq} we see that this notion coincides with all the other (equivalent) notion of modulus of the gradient given in \cite{AGS11}, and in particular there is also identification of the Sobolev spaces.

The easy part is the inclusion of the Sobolev Space obtained via relaxation of the asymptotic Lipschitz constant into the one defined by derivations. The other inclusion uses the fact that $q$-plans, namely measures on the space of curves with some integrability assumptions, induces derivations thanks to the basic observation that, even in metric spaces, we can always take the derivative of Lipschitz functions along absolutely continuous curves; this observation has already been used in \cite{Schioppa, Schioppacur, Bate} to find correlations between the differential structure of $(X,\sfd)$ and the structure of measures on the set of curves (a peculiar role is played by Alberti representation).

In Section \ref{sec:eq1} we extend this equivalence to the $BV$ space, using the results in \cite{ADM}.

\section{Sobolev spaces via derivations}\label{sec:ndef}

Here $(X, \sfd, \mm)$ will be any complete separable metric measure space, where $\mm$ is a nonnegative Borel measure, finite on bounded sets; in particular we don't put assume structural assuption, namely doubling measure nor a Poincar\'e inequality to hold. In the sequel we will denote by $\Lipk$ the set of Lipschitz functions with bounded support (the support of a continuous function $f$ is defined as ${\rm supp}(f)=\overline{ \{ f \neq 0\} }$), and with $L^0(X,\mm)$ the set of $\mm$-measurable function on $X$, without integrability assumption.

\subsection{Derivations}

We state precisely what we mean here by derivations:

\begin{definition} A derivation $\bb$ is a linear map $\bb: \Lipk \to L^0(X, \mm)$ such that the following properties hold:
\begin{itemize}
\item[(i)] (Liebniz rule) for every $f,g \in \Lipk$, we have $\bb (fg) =  \bb(f)g + f \bb(g)$;
\item[(ii)] (Weak locality) There exists some function $g \in L^0(X, \mm)$ such that
$$ | \bb (f) | (x) \leq g(x) \cdot { \rm lip}_af (x) \qquad \text{ for $\mm$-a.e. }x, \, \forall f \in \Lipk.$$
The smallest function $g$ with this property is denoted by $|b|$.
\end{itemize}
\end{definition}

From now on, we will refer to the set of derivation as ${\rm Der}(X, \sfd, \mm) $ and when we write $\bb \in L^p$ we mean $| \bb| \in L^p$. Since the definition of derivation is local on open sets we can extend $\bb$ to locally Lipschitz functions. We define also the support of a derivation ${\rm supp}(\bb)$ as the essential closure of the set $\{ |\bb| \neq 0\}$; it is easy to see that if ${\rm supp}(f) \cap {\rm supp}(\bb)=\emptyset$ then $\bb(f)$ is identically $0$.

In order to get to \eqref{eqn:weake2}, we need also the definition of divergence, and this is done simply imposing the integration by parts formula: whenever $\bb \in L^1_{\rm loc}$ we define $\div \bb$ as the operator that maps $\Lipk \ni f \mapsto - \int_X \bb (f) \, \d \mm$ (whenever this makes sense). We will say $\div \bb \in L^p$ when this operator has an integral representation via an $L^p$ function: $\div \bb =h \in L^p$ if
$$  -\int_X \bb  (f) \, \d \mm = \int_X h \cdot f \, \d \mm \qquad  \forall f \in \Lipk.$$
It is obvious that if $\div \bb \in L^p$, then is unique. Now we set
$$\Derr p (X,\sfd, \mm) = \bigl\{ \bb \in {\rm Der}(X, \sfd, \mm) \; : \; \bb \in L^p (X, \mm) \bigr \} $$
$$ \Der {p_1}{p_2} (X, \sfd, \mm) = \bigl \{ \bb \in  {\rm Der}(X, \sfd, \mm) \; : \; \bb \in L^{p_1} (X, \mm),  \div \bb\in L^{p_2} (X, \mm) \bigr\} $$
We will often drop the dependence on $(X, \sfd, \mm)$ when it is clear. We notice that ${\rm Der}$, $\Derr p$ and $\Der {p_1}{p_2}$ are real vector spaces, the last two being also Banach spaces endowed respectively with the norm $\| \bb\|_p = \| | \bb | \|_p$ and $\| \bb \|_{p_1,p_2} = \|  \bb \|_{p_1} + \|{\rm div} \, \bb \|_{p_2}$. For brevity we will denote $\Derr \infty=\Derri$, and $\Der \infty \infty =\Deri$ (b stands for bounded while L stands for Lipschitz). The last space we will consider is $D(\div)$, that will be consisting of derivation $\bb$ such that $|\bb|, \div \bb \in L^1_{\rm loc} (X, \mm)$; it is clear that $\Der pq \subseteq D(\div)$ for all $p,q \in [1, +\infty]$.

In the sequel we will need a simple operation on derivations, namely the multiplication by a scalar function: let $u \in L^0(X, \mm)$, then we can consider the derivation $u \bb$ that acts simply as $u \bb (f)  (x)= u(x) \cdot \bb (f)(x) $: it is obvious that this is indeed a derivation. We now prove a simple lemma about multiplications:

\begin{lemma}\label{lem:multi} Let $\bb$ a derivation and $u \in L^0(X,\mm)$; then we have $|u \bb |= | \bb| \cdot |u|$. Moreover, if $u \in \Lipb$ and  $\bb \in \Der{p_1}{p_2}$ we have that $ u\bb$ is a derivation such that 
$$\div ( u \bb) = u \div \bb + \bb (u) \qquad \text{ and } \qquad u \bb \in \Der {p_1}{p_3},$$
where $p_3=\max \{p_1,p_2\}$; in particular we have that $\Der pp$ is a $\Lipb$-module.
\end{lemma}

\begin{proof} Let us prove the first assertion: it is clear that $|u \bb | (x)\leq |\bb|(x) \cdot |u(x)|$ by definition; the other inequality is obvious in the set $\{ u =0 \}$. In order to prove the converse inequality also in $\{ u \neq 0\}$ we can choose $\bb_u= u \bb$ and let
$$ g(x)= \begin{cases} u^{-1}(x) \qquad & \text{ if }u(x) \neq 0 \\ 0 & \text{ otherwise.}\end{cases} $$
Then we know that $| g \bb_u | \leq |g| \cdot |\bb_u|$. Noting that $\bb(f)=g\bb_u(f)$ in $\{u \neq 0\}$ for every $f \in \Lipb$, we get also $| \bb| = | g \bb_u |$ in the same set and so 
$$|\bb|=|g \bb_u| \leq |g| \cdot |\bb_u| \leq |g| \cdot |u| \cdot |\bb|=|\bb| \qquad  \text{ in }\{u\neq 0\}.$$
In particular we have $ | \bb_u| = |\bb| \cdot |u|$ in $\{u\neq 0\}$ and thus the thesis.

For the second equality we can use Liebniz rule: let $f \in \Lipb$, and using $\bb (fu) = u \bb(f) + f\bb(u)$
\begin{align*} we have
   - \int_X u \bb(f)  \, \d \mm  &=  - \int_X \bb (fu) \, \d \mm+ \int_X f \bb(u) \, \d \mm \\
 & = \int_X f u \cdot \div \bb  \, \d \mm +\int_X f \bb(u) \, \d \mm\\
 &= \int_X f \cdot ( u \div \bb + \bb (u) ) \, \d \mm
 \end{align*}
 and so, thanks to the arbitrariness of $f$ we get $\div (u \bb) =  u \div \bb + \bb (u) $. 
\end{proof}

\begin{lemma}[Stong locality in $D(\div)$] Let $\bb \in D(\div)$. Then for every $f,g \in {\rm Lip}(X,\sfd)$ we have
\begin{itemize}
\item[(i)] $\bb(f)=\bb(g)$ $\mm$-almost everywhere in $\{f=g\}$;
\item[(ii)] $\bb(f) \leq |\bb| \cdot {\rm lip_a} (f|_C)$ $\mm$-almost everywhere in $C$, for every closed set $C$;
\end{itemize}
\end{lemma}

\begin{proof} In order to prove (i), it is sufficient to consider $g=0$ and $f$ with support contained in $B=B_r(x_0)$, where we can take $r>0$ as small as we want; then we can conclude by linearity and weak locality. So we can suppose that both $|\bb|$ and $\div \bb$ are integrable in $B$. Now we can consider $\phi_{\ep}(x)= (x-\ep)_+ - (x+\ep)_-$; we have $\phi_{\ep}$ is a $1$-Lipschitz function such that $|\phi_{\ep}(x)-x| \leq \ep$ and $\phi(x)=0$ whenever $|x|\leq \ep$. Let $f_{\ep} = \phi_{\ep}(f)$; we have $\bb(f_{\ep})$ is a family of equi-integrable functions and so there is a subsequence converging weakly in $L^1$ to some function $g$. Moreover $f_{\ep} \to f$ uniformly and in particular 
\begin{equation}\label{eqn:convergence} \int_X \bb (f_{\ep})\, \d \mm - \int_X \bb(f)\, \d \mm = - \int_X (f_{\ep}- f)\cdot \div \bb \, \d \mm  \to 0; \end{equation}
since this is true also for $\rho \bb$ whenever $\rho \in \Lipk$ (by Lemma~\ref{lem:multi}, we have $\rho \bb \in D(\div )$), we obtain $\langle \rho, \bb(f_{\ep}) \rangle \to \langle \rho, \bb(f) \rangle $ for every $\rho \in \Lipb$ and so $g=\bb(f)$ and $\bb(f_{\ep}) \weakto \bb(f)$. In particular, letting $\rho= \chi_{\{f=0\}} {\rm sgn} (\bb(f))$ and noting that ${\rm lip}_a (f_{\ep} )=0$ in the set $\{ |f| < \ep \}$ we obtain
$$\int_{\{f=0\}} |\bb (f)| \, \d \mm = \int_X \rho \cdot \bb (f) = \lim_{\ep \to 0} \int_X \rho \cdot \bb(f_{\ep}) \, \d \mm =0.$$

For (ii) we proceed as follows: for every closed ball $\bar{B}_r(y)$ we consider the McShane extension of the function $f$ restricted to $C \cap \bar{B}_r(y)$ and we call it $g^r_y$:
$$g^r_y(x)= \sup \{ f(x')- L\sfd(x',x)\; :\; x' \in C \cap \bar{B}_r(y)\}, \qquad L={\rm Lip}(f, \bar{B}_r(y) \cap C)$$
In particular we have $f=g^r_y$ on $C \cap B_r(y)$ and ${\rm Lip}(g^r_y , B_r(y)) = {\rm Lip}(f, B_r(y) \cap C) = {\rm Lip}(f|_C, B_r(y))$. Applying (i) of this lemma we find that $\bb(f)=\bb(g^r_y)$ $\mm$-a.e. on $C \cap \bar{B}_r(y)$; in particular
$$ |\bb(f) | (x) \leq |\bb| \cdot {\rm Lip} (f|_C, \bar{B}_r(y) ) \qquad \mm \text{-a.e. on } C \cap B_r(y).$$
Since we have $\bar{B}_r(y) \subset B_{2r}(x)$ whenever $x \in B_r(y)$, we obtain
$$ |\bb(f) | (x) \leq |\bb| \cdot {\rm Lip} (f|_C, B_{2r}(x) ) \qquad \mm \text{-a.e. on } C \cap B_r(y);$$
now we can drop the dependance on $y$ and then let $r \to 0$ to get the thesis.
\end{proof}

\begin{remark} Notice that our definition of derivation is slightly different from the classical one of Weaver \cite{W00}, since we don't require any continuity assumption. However it is easy to see that every derivation in $D(\div)$ is also a Weaver derivation, thanks to the integration by part formula. In the sequel only these derivations will play a role in the definition of Sobolev Spaces and so this discrepancy in the definition is harmful.
\end{remark}

\subsection{Definition via derivations}

In this whole section we treat the Sobolev spaces $W^{1,p}$ with $1\leq p<+\infty$; the case of the space $BV$ will be treated separately. We state here the main definition of Sobolev space via derivations: we want to follow the definition \eqref{eqn:weake2} but in place of the scalar product between the vector field and the weak gradient we assume there is simply a continuous linear map.

\begin{definition}\label{def:deri}  Let $f \in L^p(X,\mm)$; then $f \in W^{1,p}(X, \sfd, \mm)$ if, setting $p=q/(q-1)$, there exists a linear map $L_f: \Der qq \to L^1(X,\mm)$ satisfying
\begin{equation}\label{eq:defw}
\int_X L_f(\bb)\,d\mm=-\int_X f \div \bb\,d\mm\qquad\text{for all $\bb\in \Der qq$,}
\end{equation}
continuous with respect to the $\Derr q$ norm and such that $L_{f}(h \bb)=hL_f(\bb)$ for every $h \in {\rm Lip}_b, \bb \in \Der qq$. When $p=1$ assume also that $L_f$ can be extended to an $L^{\infty}$-linear map in $\Deri^{\infty}:=L^{\infty} \cdot \Deri$.
\end{definition}

Since from the definition it is not obvious, we prove that $L_f(\bb)$ is uniquely defined whenever $f \in W^{1,p}$ and $\bb \in \Der qq$:

\begin{remark}[Well posedness in $\Der qq$] Let us fix $\bb \in \Der qq$, $f \in W^{1,p}$; let $L_f$ and $\tilde{L}_f$ be two different linear maps given in the definition on $W^{1,p}$. Let $h \in \Lipb$: using Lemma \ref{lem:multi} we have $ h \bb \in \Der qq$ and so we can use \eqref{eq:defw} and the $L^{\infty}$-linearity to get
$$ \int_X h L_f(\bb) \, \d \mm = \int_X L_f (h \bb)= -\int_X f \div ( h\bb) \, \d \mm,  $$
and the same is true for $\tilde{L}_f$. In particular, since the right hand side does not depend on $L_f$, we have $\int_X h L_f( \bb) = \int_X h \tilde{L}_f (\bb)$, and thanks to the arbitrariness of $h \in \Lipb $ we conclude that $L_f( \bb) = \tilde{L}_f( \bb)$ $\mm$-a.e. We will call this common value $\bb(f)$, since it extends $\bb$ on Lipschitz functions. The same result is true also for $p=1$ and $\bb \in \Derii$.
\end{remark}
Now we can give the definition of weak gradient, in some sense dual to the definition of $|\bb|$:

\begin{theorem}\label{thm:improved} Let $f \in W^{1,p}(X, \sfd, \mm)$; then there exists a function $g_f \in L^p(X, \mm)$ such that
\begin{equation}\label{eqn:improvedth}
| \bb(f)| \leq g_f \cdot | \bb|  \quad \mm \text{-a.e. in }X \; \;\qquad \forall \bb \in \Der qq.
\end{equation}
The least function $g_f$ (in the $\mm$-almost everywhere sense) that realizes this inequality is denoted with $| \nabla f |_p$, the $p$-weak gradient of $f$
\end{theorem}

%Before proving this theorem we give a little weaker definition in the spirit of \eqref{eqn:weake2} that will turn out to be equivalent to this first:
%
%\begin{definition} Let $f \in L^p(X, \sfd,\mm)$; then we say that $f$ is variationally Sobolev and $f \in W_v^{1,p}(X, \sfd, \mm)$ if there exists a constant $C$ such that
%$$ \left| \int_X  f \cdot \div \bb \, \d \mm \right| \leq C  \|  \bb  \|_{L^q( \mm )} \qquad \forall \bb \in \Der qq. $$
%The minimal constant is denoted by $C_f$.
%\end{definition}
%It is clear that a function in $W^{1,p}$ belongs also to $W^{1,p}_v$. The other implication will follow from the equivalence theorem. Now we prove the existence of a weak gradient, in the integral sense:

%\begin{proposition}\label{prop:existence} If $f \in W^{1,p}$ then there exists a function $g \in L^p(X, \sfd ,\mm)$ such that the following inequality holds:
% \begin{equation}\label{def:wgrd} \int_X f \cdot \div \bb \, \d \mm \right  \leq \int_X g | \bb |\, \d \mm  \qquad \forall \bb \in \Der qq. 
% \end{equation}
%The unique $g$ with minimal norm satisfying \eqref{def:wgrd} is denoted by $| df|_w$. Moreover we have that $\| |df|_w\|_p = C_f$.
%\end{proposition}

\begin{proof} We reduce to prove the existence of a weak gradient in the integral sense; then thanks to $ {\rm Lip}_b$-linearity we can prove the theorem. In fact if we find a function $g \in L^p(X, \mm)$ such that
 \begin{equation}\label{def:wgrd} 
\int_X  \bb (f )\, \d \mm   \leq \int_X g | \bb |\, \d \mm  \qquad \forall \bb \in \Der qq, 
 \end{equation}
then, choosing $\bb_h =h \bb$ with $h \in \Lipb$, we can localize the inequality thus obtaining $\bb(f) \leq g| \bb|$; using this inequality also with the derivation $-\bb$ we get \eqref{eqn:improvedth}. 

So, we're given a function $f \in W^{1,p}$ and we want to find $g \in L^p$ satisfying \eqref{def:wgrd}; let us note that, by definition, %there exists a continuous linear map  $L_f: \Derr q \to L^1$ that extends $\bb(f)$ and so there exists a constant $C= \| L_f \|$ such that
there exists a constant $C=\| L_f\|$ such that for every $\bb \in \Der qq$
\begin{equation}\label{eqn:ovvio1}
 \int \bb(f) \, \d \mm \leq \| L_f ( \bb ) \|_1 \leq C \| \bb \|_q
 \end{equation}
Let us consider two functionals in the Banach space $Y=L^q(X,\mm)$:
\begin{equation}\label{eqn:psi1}
 \Psi_2 ( h) = C \| h \|_{L^q(\mm)}
\end{equation}
\begin{equation}\label{eqn:psi2}
 \Psi_1 ( h) = \sup \left\{ \int_X\bb (f) \, \d \mm : | \bb| \leq h  \; , \; \bb \in \Der qq \right\}
\end{equation}
where the supremum of the empty set is meant to be $- \infty$. Equation \eqref{eqn:ovvio1} guarantees that
\begin{equation}\label{eqn:ovvio2}
 \Psi_1( h) \leq \Psi_2 (h) \qquad \forall h \in Y.
\end{equation}
Moreover $\Psi_2$ is convex and continuous while we claim that $\Psi_1$ is concave: it is clearly positive $1$-homogeneus and so it is sufficient to show that
$$\Psi_1( h_1 + h_2 ) \geq \Psi_1(h_1) + \Psi_1(h_2). $$
We can assume that $\Psi_1(h_i) > - \infty$ for $i=1,2$ because otherwise the inequality is trivial. In this case for every $\ep>0$ we can pick two derivations $\bb_i \in \Der qq$ such that
$$ \int_X \bb_1(f) \, \d \mm \geq \Psi_1 (h_1) - \ep  \qquad | \bb_1 | \leq h_1$$
$$ \int_X \bb_2(f)  \, \d \mm \geq \Psi_1 (h_2) - \ep  \qquad |\bb_2 | \leq h_2 $$
and so we can consider $\bb_1 + \bb_2$ that still belongs to $\Der qq$ and clearly $|\bb_1 + \bb_2 | \leq |\bb_1 |+ | \bb_2 | \leq (h_1+h_2)$ and so
$$ \Psi_1 ( h_1 + h_2) \geq \int_X (\bb_1+\bb_2)(f)  \, \d \mm\geq \Psi_1 (h_1) + \Psi_1(h_2) - 2 \ep, $$
and we get the desired inequality letting $\ep \to 0$.
By Hahn-Banach theorem we can find a continuous linear functional $L$ on $L^q(X, \mm)$ such that
$$ \Psi_1(h) \leq L(h) \leq \Psi_2(h). $$

\textbf{Case  $p>1$.} We know that $(L^q)^*=L^p$ and so we can find $g \in L^p$ such that $L(h)= \int_X gh \,  \d \mm$. This proves the existence and moreover we have that $L(h) \leq \Psi_2(h) = C \|h\|_q$ for every $h \in Y$ and so we have also that $\|g\|_p  \leq C$. 

\textbf{Case $p=1$, $X$ compact.} In this case (notice that here we have to put $ \Deri^{\infty}$ in place of $\Der qq$ in \eqref{eqn:psi2}) if we restrict $L : C_b(X) \to \R$ we can see it as a positive linear such that $L(h) \leq C \| h \|_{\infty}$ and so, thanks to the compactness of $X$, it can be represented as a finite measure, i.e. there exists $\mu \in \mathcal{M}_+(X)$ such that $L(h) = \int_{X} h \, \d \mu$ for every $h \in C_0(X)$ and $\mu(X) \leq C$. Now let us fix $\bb \in \Derii$ and let 
$$h_{\ep} (x)= \begin{cases}  \frac  1 { |\sbb|}  \qquad & \text{ if }|\bb|(x)\geq \ep \\ \ep^{-1} &  \text{ otherwise}  \end{cases}$$
in such a way that $|h_{\ep}\bb| \leq1$ with equality in $\{ | \bb| \geq \ep \}$. Now let us consider for every $h \in C_0(X)$ the derivation $ h \cdot h_{\ep} \cdot \bb$; we know that $|h \cdot h_{\ep} \cdot \bb | \leq |h| $ and so we can use \eqref{eqn:psi2} and the $L^{\infty}$-linearity to infer that
$$\int_X h h_{\ep} \bb(f) \, \d \mm \leq \int_X |h| \, \d \mu \qquad \forall h \in C_0(X);$$
this permits us to localize the inequality to $h_{\ep}\bb(f) \mm \leq \mu$. Now we have a family of measures $ \mathcal{F} = \{ h_{\ep}\bb(f) \mm \; : \; \forall \bb \in \Derii ,  \, \forall \ep > 0 \}$ such that $\nu \leq \mu$ whenever $\nu \in \mathcal{F}$. Now we can consider the supremum of the measures in $\mathcal{F}$, defined as
$$\mu_{\mathcal{F}} (A) = \sup \Bigl\{ \sum_{i=1}^N \nu_i(A_i) \; : \; \nu_i \in \mathcal{F} , \, \bigcup A_i \subseteq A ,\; A_i \text{ disjoint }\Bigr\}; $$
it is readily seen that this is in fact a measure, and it is the least measure $\rho$ such that $\nu \leq \rho$ for every $\rho \in \mathcal{F}$. The existence is clear thanks to the fact that $\nu \leq \mu$, and in particular we have that $\mu_{\mathcal{F}} \leq \mu$;  moreover, since for every $\nu \in \mathcal{F}$ we have that $\nu << \mm$, also the supremum inherits this property, in particular we have $\mu_{\mathcal{F}} = g \mm$ for some $g \in L^1(\mm)$. In particular, again fixing $\bb \in \Derii$, we have that 
\begin{equation}\label{eq:quasifinito}
h_{\ep} \bb(f) \leq g \quad \mm \text{-a.e.} \qquad \forall \ep>0;
\end{equation}
in particular, we can divide \eqref{eq:quasifinito} by $h_{\ep}$ to obtain
\begin{equation}\label{eq:finito}
\begin{cases} \bb(f) \leq g | \bb | \qquad &\mm \text{-a.e. in } \{ |\bb| \geq \ep \} \\
\bb(f) \leq g \ep \qquad &\mm \text{-a.e.  in }\{| \bb| < \ep \}.
\end{cases}
\end{equation}
Since $\ep$ is arbitrary we obtain $\bb(f) \leq g |\bb|$ for $\mm$-almost every $x \in X$, that is the thesis; also in this case $p=1$ we have $\| g\|_{1} \leq \mu(X) \leq \| L_f \|$.

\textbf{ Case $p=1$, $X$ general.} In order to remove the compactness assumption, for every compact non negligible set $K \subseteq X$ let us consider the two functionals in the Banach space $Y_K=L^{\infty}(K, \sfd,\mm)$:
\begin{equation}\label{eqn:psi1k}
 \Psi_2 ( h) = C \| h \|_{L^{\infty}(K,\mm)}
\end{equation}
\begin{equation}\label{eqn:psi2k}
 \Psi_1 ( h) = \sup \left\{ \int_K \bb(f) \, \d \mm : | \bb| \leq h \quad  \mm \text{-a.e. on }K \; , \; \bb \in \Derii \right\}.
\end{equation}
Now we can argue precisely as before to obtain $g_K \in L^1(K,\mm)$ such that $\|g_K \|_1 \leq \|L_f \|$
\begin{equation}\label{eq:defg}
\bb(f) \leq g_K | \bb| \quad \mm \text{-a.e. on }K \qquad \forall \bb \in \Derr q.
\end{equation}
Now for every increasing sequence of compact sets $K_n$, let us consider $g(x)= \inf_{ K_n \ni x } g_{K_n} (x)$. Denoting $ Y:= \bigcup_n K_n$, it is easy to note that $g \in L^1(Y,\mm)$, since $\| g \|_{L^1(Y, \mm) } = \sup_n \| g \|_{L^1(K_n, \mm) } \leq  \sup_n \| g_{K_n} \|_{L^1(K_n, \mm) } \leq  \| L_f \|$, and we have that
$$ \bb(f) \leq g | \bb| \quad \mm \text{-a.e. on }Y \qquad \forall \bb \in \Derr q;$$
so, in order to conlcude, it is sufficient to find a sequence $K_n$ such that $\mm( X \setminus \bigcup_n K_n )=0 $, but this can be done thanks to the hypothesis of $\mm$ finite on bounded sets (so we can find $\theta >0$ such that $\theta \mm$ is finite and then apply Prokhorov theorem to $\theta \mm$).

\end{proof}

\section{Equivalence with other definitions}\label{sec:eq}

In this section we want to prove, when $p>1$, that Definition \ref{def:deri} is equivalent to the other ones $W_*^{1,p}$ and $W_w^{1,p}$, given in \cite{AGS11}. As a byproduct we obtain the equivalence also with other definitions of Sobolev Spaces, for example the one given in \cite{Cheeger00}, similar to $W^{1,p}_*$ but here the relaxation is made with general $L^p$ functions, and the asymptotic Lipschitz constant is replaced by upper gradients, or the one given in \cite{Shanmugalingam00}, similar to $W^{1,p}_w$ but with a slightly stronger notion of negligibility of set of curves.

 We will prove that $W_*^{1,p} \subseteq W^{1, p} \subseteq W_w^{1,p}$ and that the following inequality is true for the weak gradients:
$$| \nabla f |_{p,*} \leq | \nabla f |_p \leq | \nabla f|_{p,w} \qquad \mm \text{-a.e. in } X$$
Then for $p>1$, using the equivalence $W_*^{1,p} = W_w^{1,p}$ and $| \nabla f|_{p,w}  = |\nabla f |_{p,*}$ in \cite{AGS11} will let us conclude; also the coincidence with other definitions can be found in \cite{AGS11}. Let us recall briefly the definitions of $W^{1,p}_*$ (in the stronger version given in \cite{ACDM}) and $W^{1,p}_w$:

\begin{definition}[Relaxed Sobolev Space] A function $f \in L^p(X, \mm)$ belongs to $W^{1,p}(X, \sfd, \mm)$ if and only if there exists a sequence $(f_n) \subset \Lipk $ and a function $g \in L^p(X, \mm)$ such that
$$ \lim_{n \to \infty} \| f_n - f \|_p + \| {\rm lip}_a (f_n )- g\|_p =0. $$
The function $g$ with minimal $L^p$ norm that has this property will be denoted with $|\nabla f|_{p,*}$
\end{definition}

In order to define the space $W^{1,p}_w$ we have to introduce the test plans, that will consent to define a concept of negligibility of set of curves that is crucial in the definition of the weak Sobolev space (see also \cite{ADMS} for a detailed analysis of the different concepts of negligibility given in \cite{AGS11} and \cite{Shanmugalingam00}) .

\begin{definition}[Test plans and negligible sets of curves]\label{def:testplans}
We say that a probability measure $\ppi\in\prob{C([0,1],X)}$ is a
$q$-\emph{test plan} if $\ppi$ is concentrated on $AC^q([0,1],X)$, we have
$\iint_0^1|\dot\gamma_t|^q\d t\,\d\ppi<\infty$ and there exists a
constant $C(\ppi)$ such that
\begin{equation}
(\e_t)_ \sharp\ppi \leq C(\ppi)\mm\qquad\forall t\in[0,1].
\label{eq:1}
\end{equation}
A set $A\subset C([0,1],X)$ is said to be
$p$-\emph{negligible} if it is contained in a $\ppi$-negligible set for any $p$-test plan $\ppi$. 
A property which holds for every $\gamma\in C([0,1],X)$, except
possibly a $p$-negligible set, is said to hold for $p$-almost every
curve.
\end{definition}

\begin{definition}[Weak Sobolev Space] A function $f \in L^p(X, \mm)$ belongs to $W^{1,p}_w(X, \sfd, \mm)$ if there exists a function $g \in L^p(X, \mm)$ that is a $p$-weak upper gradient, i.e. it is such that
\begin{equation}
\label{eq:inweak} \left|\int_{\partial\gamma}f\right|\leq
\int_\gamma g<\infty\qquad\text{for $p$-a.e. $\gamma$.}
\end{equation}
The minimal $p$-weak upper gradient (in the pointwise sense) will be denoted by $| \nabla f |_{p,w}$.
\end{definition}

\subsection{ $W^{1,p}_* \subseteq W^{1,p}$}\label{ss:1eq}

Let $f \in W^{1,p}_*(X,\sfd,\mm)$. Then, by definition, there exists a sequence of Lipschitz functions with bounded support such that $f_n \stackrel{p}{\to} f$ and ${\rm Lip}_a (f_n) \stackrel{p}{\to} | \nabla f|$. Then by the weak locality property of derivations and the definition of divergence have that for every $\bb \in \Der qq$
$$\left| \int_X f_n \cdot \div \bb \, \d \mm \right| = \left|  \int_X \bb(f_n ) \, \d \mm \right| \leq \int_X | \bb | \cdot {\rm lip}_a (f_n) \, \d \mm. $$
Taking the limit as $n \to\infty$ we have that 
\begin{equation} \label{eq:ineqwg} \left | \int_X f \cdot \div \bb \, \d \mm \right| \leq \int_X | \bb| \cdot | \nabla f |_{p,*} \, \d \mm \qquad \forall \bb \in \Der qq.
\end{equation}
Now we have to construct the linear functional $L_f : \Der qq \to L^1$. %We note that it is sufficient to construct this functional on $\Der qq$ and then extend it to $\Derr q$ via Lemma \ref{lem:hbl}, applied to $W= \Derr q$ and $V=L^{\infty} \cdot \Der qq$, the $L^{\infty}$-module generated by $\Der qq$.
So, fix $\bb \in \Der qq$ and let $\mu_{\sbb}= | \bb | \cdot | \nabla f|_{p,*} \mm$. Notice that $\mu_{\sbb}$ is a finite measure. Now let $R^\sbb: \Lipb \to \R$ be the linear functional defined by

$$ R^{\sbb}(h) =  - \int_X f \cdot \div (h \bb) \, \d \mm.$$
Notice that, since $h\bb \in \Der qq$, we can take it as a test derivation in \eqref{eq:ineqwg}, obtaining $| R^{\sbb}(h) | \leq C \|h \|_{\infty}$, where $C=\mu_{\sbb}(X)$. In particular $R^{\sbb}$ can be extended to a continuous linear functional on $C_b(X)$; since $|R^{\sbb}(h)| \leq \int_X | h | \, \d \mu_{\sbb}$, we have that $R^\sbb(h)$ can be represented as an integral with respect to a signed measure $\mm_\sbb$, whose total variation is less then $\mu_{\sbb}$, but since $\mu_\sbb$ is absolutely continuous with respect to $\mm$, also $\mm_\sbb$ must have this property; if we denote by $L_f(\bb)$ the density of $\mm_\sbb$ relative to $\mm$, we have
\begin{align}\label{eq:1}
- \int_X f \cdot \div (h\bb) \, \d \mm& = \int h \cdot  L_f(\bb) \, \d \mm\qquad \forall \, h \in \Lipb \\
\label{eq:2}
 \qquad |L_f(\bb)| &\leq |\bb| \cdot | \nabla f|_{p,*}  \qquad \text{ $\mm$-almost everywhere }
 \end{align}
 Now we have to check the ${\rm Lip}_b$-linearity, but this is easy since for every $h, h_1 \in {\rm Lip}_b$ by definition we have $R^{h\sbb}(h_1)=R^{\sbb}(h \cdot h_1)$; in particular
 $$\int_X h_1 \cdot L_f(h \bb) \, \d\ mm = \int_X h_1 \cdot h L_f(\bb) \, \d \mm \qquad \forall h_1 \in {\rm Lip}_b(X),$$
 and so $L_f(h \bb) = h L_f(\bb)$.
%Now we have to check the $L^{\infty}$-linearity. Let us suppose that for some $k \in L^{\infty}$ we have $\bb, k \bb \in \Der qq$. We can find a sequence $(h_n) \subset \Lipb$ converging to $k$ in $L^1(\mu_{\sbb})$. Now fix a function $h \in \Lipb$: then using \eqref{eq:2} we get
%$$\left|  \int h h_n  \cdot L_f(\bb) \, \d \mm   - \int h k \cdot L_f(\bb) \, \d \mm\right | \leq   \int h |h_n - k| |L_f (\bb)| \,  \d \mm \leq \| h \|_{\infty} \| h_n - k \| _{L^1(\mu_\sbb)}  \to 0; $$
%and so $R^\sbb (h h_n) \to \int_X hk \cdot \bb (f) $. Since $(h_n- k)h \bb \in \Der qq$, we have also
%$$ |R^{h_n\sbb} (h) - R^{k \sbb} (h ) | = \left| \int_X f \cdot \div ( (h_n-k)h \bb )  \, \d \mm \right| \leq \| h \|_{\infty} \| h_n - k \|_{L^1 (\mu_\sbb) }$$
%and, by \eqref{eq:1} we have $R^{h_n\sbb}(h) = \int_X  h_n h \cdot L_f(\bb) \, \d \mm$ and so we can conclude
%$$ \int_X hk \cdot L_f(\bb) \, \d \mm = \int_X h \cdot L_f( k \bb) \, \d \mm  \qquad \forall \, h \in \Lipb;$$
%for the arbitrariness of $h$ we deduce $L_f( k \bb) = k L_f( \bb)$ as desired. Similar reasoning applies to show that $L_f(k_1 \bb_1 + \ldots  + k_n \bb_n) =k_1 L_f( \bb_1) + \ldots + k_n L_f(\bb_n)$ whenever $\bb_i \in \Der qq$, $k_i \in L^{\infty}(X, \mm)$ and $k_1 \bb_1 + \ldots  + k_n \bb_n  \in \Der qq$.

\subsection{$W^{1,p} \subseteq W^{1,p}_w$}\label{ss:2eq}

The crucial observation is that every $q$-plan induce a derivation:

\begin{proposition}\label{prop:pipi} Let $\ppi$ be a $q$-plan. For every function $f \in \Lipb$ let us consider $\bb_{\sppi}(f)$, the function such that:
%$\bb_{\sppi} \in \Der qq$ such that:
\begin{equation}\label{def:bspi}
 \int_X g \cdot \bb_{\sppi} (f ) \, \d \mm = \int_{AC} \int_0^1 g(\gamma_t) \frac {d ( f \circ \gamma )}{ds} (t) \, \d t \, \d \ppi (\gamma) \qquad \forall g  \in L^p. 
 \end{equation}
 Then we have that $\bb_{\sppi} \in \Der qq$ and moreover
 \begin{equation}\label{def:modul}
 \int_X g \cdot| \bb_{\sppi}| \, \d \mm \leq \iint_{\gamma} g  \, \d s \, \d \ppi (\gamma) \qquad \forall g  \in L^p , g \geq 0;
 \end{equation}
 \begin{equation}\label{def:bspi2}
 \int_X f  \cdot \div( \bb_{\sppi}) \, \d \mm = \int_{AC} (f(\gamma_1) - f(\gamma_2)) \, \d \ppi (\gamma) \qquad \forall f  \in L^p. 
 \end{equation}
\end{proposition}

\begin{proof} We first fix $f \in \Lipb$ and notice that the right hand side in \eqref{def:bspi} is well defined thanks to Rademacher theorem. Then the Liebniz rule is easy to check thanks to its validity in the right hand side of \eqref{def:bspi}. In order to find a good candidate for $|\bb_{\sppi}|$, we estimate $\frac {d ( f \circ \gamma )}{ds} \leq {\rm lip}_a (f )( \gamma_t ) |\dot{\gamma_t}|$ and so, for every nonnegative $g \in L^p$ we have
%
%\begin{align*}
% \left| \int_{AC} \int_0^1 g(\gamma_t) \frac {d( f \circ \gamma) }{ds} (t) \, \d t \, \d \ppi (\gamma) \right| &\leq {\rm Lip}(f) \int_{AC} \int_0^1 |g|(\gamma_t) | \dot{\gamma_t}| \, \d t \, \d \ppi \\
%& \leq {\rm Lip}(f) \left( \iint |g|(\gamma_t)^p \, \d t \d \ppi \right)^{1/p}  \left( \iint |\dot{\gamma_t}|^q \, \d t \d \ppi \right)^{1/q} \\
%& \leq {\rm Lip}(f) \cdot C(\ppi)^{1/p} \cdot \| g\|_{L^p(\mm)} \cdot \| E_q( \gamma) \|_{L^q( \sppi)},
%\end{align*}
%so that, by duality, there exists a function $\bb_{\sppi}(f) \in L^q(X, \mm)$ that realizes \eqref{def:bspi}. It is easy to prove that this is indeed a derivation, exploiting the Liebniz formula and the linearity of the derivative in $\R$ in the right hand side of \eqref{def:bspi}.
%In order to prove also an estimate for $| \bb_{\sppi}|$ we note that  
$$\int_0^1 g(\gamma_t)  \frac {d ( f \circ \gamma )}{ds} (t) \, \d t \leq \int_0^1 g(\gamma_t) {\rm lip}_a (f) (\gamma_t) | \dot{\gamma}_t| \, \d t;$$ 
integrating with respect to $\ppi$ and using Fubini theorem we get
\begin{equation}\label{eq:cc}
 \int_X g \cdot \bb_{\sppi} (f) \, \d \mm \leq  \int_X g\cdot  {\rm lip}_a (f)  \, \d \mu_{\sppi},
\end{equation}
where $\mu_{\sppi} = \int_0^1 (\e_t)_{\sharp} ( \|\dot{\gamma}_t | \ppi ) \, \d t$ is the \emph{barycenter} of $\ppi$, and it is such that
\begin{equation}\label{eq:mu}
 \int_X g \, \d \mu_{\sppi} = \iint_{\gamma} g \, \d s \, \d \ppi.
 \end{equation}
In particular we can use H\"older's inequality to estimate the behavior of $ \mu_{\sppi}$:
\begin{align*}
 \int_X g \, \d \mu_{\sppi}  & =   \int_{AC} \int_0^1 g(\gamma_t) | \dot{\gamma_t}| \, \d t \, \d \ppi \\
& \leq \left( \iint |g(\gamma_t)|^p \, \d t \d \ppi \right)^{1/p}  \left( \iint |\dot{\gamma_t}|^q \, \d t \d \ppi \right)^{1/q} \\
& \leq  C(\ppi)^{1/p} \cdot \| g\|_{L^p(\mm)} \cdot \| E_q( \gamma) \|_{L^q( \sppi)},
\end{align*}
and so, by duality argument, we obtain that $\mu_{\sppi} = h \mm$ with $h \in L^q(X,\mm)$; using this representation in \eqref{eq:cc} we obtain
$$ \int_X  g \cdot \bb_{\sppi} (f) \, \d \mm \leq \int_X g \cdot {\rm lip}_a (f) h \, \d \mm \qquad \forall \, g\in L^q , \; g \geq 0.$$
So we deduce that $|\bb_{\sppi}|\leq h$ and in particular $\bb_{\sppi} \in L^q$ and \eqref{def:modul} is true thanks to \eqref{eq:mu}.

It remains to prove the last equality: by definition of divergence we have, for $f \in \Lipk$
\begin{equation}\label{eq:verlip}
 \int f \cdot \div \bb_{\sppi} \, \d \mm = \int_{AC} \int_0^1  \frac {d ( f \circ \gamma )}{ds} (t) \, \d t \, \d \ppi (\gamma) = \int ( f(\gamma_1) - f(\gamma_0)) \, \d \ppi,
\end{equation}
thanks to the fact that the fundamental theorem of calculus holds for Lipschitz functions. By definition of $q$-plan we have also that $(\e_t)_{\sharp} \ppi =f_t \mm $ where $f_t \leq C(\ppi)$ for every $t \in [0,1]$; since $\ppi$ is a probability measure we have $\int f_t \, \d \mm=1$ and so $f_t \in L^1 \cap L^{\infty}$ and in particular $f_t \in L^q$ and so $\div \bb_{\sppi}= (f_1 - f_0 ) \in L^q$. This enables us to extend \eqref{eq:verlip} to $f \in L^p$ and so we proved also \eqref{def:bspi2}.
\end{proof} 

% and, using that $\int_{\gamma} (f \circ \gamma)' \, \d t = f(\gamma_1) - f(\gamma_2)$ and the definition of divergence we obtain  $\div \bb_{\sppi} = (\e_1)_{\sharp} \ppi - (\e_0)_{\sharp} \ppi$ and so $\div \bb \in L^q$ since $\ppi$ is a $q$-plan. We can try to estimate $|\bb|$ simply using
%$$\int_0^1 g(\gamma_t)  \frac {d ( f \circ \gamma )}{ds} (t) \, \d t \leq \int_0^1 g(\gamma_t) {\rm lip}_a (f) (\gamma_t) | \dot{\gamma}_t| \, \d t $$ 
%
%\begin{equation}
%\label{eqn:normbb}
%| \bb_{\sppi}| \leq \int_0^1 (\e_t)_{\sharp} ( \|\dot{\gamma}_t | \ppi ) \, \d t; 
%\end{equation}
%In this way it is easy to see that also $\|\bb_{\sppi}\|_q$ is finite. In fact, given a function $g \in L^p( \mm)$ we have that
%
%\begin{align*} 
%\int g |\bb_{\sppi}| \,  \d \mm & \leq \int_{AC} \int_0^1 g(\gamma_t) | \dot{\gamma_t}| \, \d t \, \d \ppi \\
%& \leq \left( \iint g(\gamma_t)^p \, \d t \d \ppi \right)^{1/p}  \left( \iint |\dot{\gamma_t}|^q \, \d t \d \ppi \right)^{1/q} \\
%& \leq C(\ppi)^{1/p} \cdot \| g\|_{L^p(\mm)} \cdot \| E_q( \gamma) \|_{L^q( \sppi)}.
%\end{align*}
%
%

\begin{lemma} Let $f \in W^{1,p}(X, \sfd, \mm)$. Then $|\nabla f|_{w}$ is a $p$-weak upper gradient for $f$.
\end{lemma}

\begin{proof}
By Proposition \ref{prop:pipi} we know that for every $q$-plan $\ppi$ we can associate a derivation $\bb_{\sppi} \in \Der qq$; we use this derivation in the definition of $W^{1,p}$ and, using also Theorem \ref{thm:improved}, we obtain
$$- \int_X f \cdot \div \bb_{\sppi} \, \d \mm  \leq \int | \nabla f|_w  \cdot |\bb_{\sppi}| \, \d \mm;$$
Now, using \eqref{def:modul} and \eqref{def:bspi2}, we obtain precisely

\begin{equation}
\label{eqn:almost}
\int_{AC} (f(\gamma_0) - f(\gamma_1)) \d \ppi \leq \int_{AC} \int_\gamma |\nabla f|_w\, \d s \,\d \ppi. \qquad \forall \ppi \text{ $q$-plan}
\end{equation}

We can "localize" this inequality using the fact that for every Borel set $A \subseteq C([0,1];X)$ such that $\ppi(A) \neq 0$, we have that $\ppi_A=\frac 1{\sppi(A)} \ppi|_A$ is still a $q$-plan and so we can infer that 
\begin{equation}
\label{eqn:almost}
\int_{A} (f(\gamma_0) - f(\gamma_1)) \d \ppi \leq \int_{A} \int_\gamma |\nabla f|_w \, \d s\,\d \ppi. \qquad \forall A \subset C([0,1];X),
\end{equation}
and so $f(\gamma_0) - f(\gamma_1) \leq  \int_{\gamma} | \nabla f |_w$ for $\ppi$-almost every curve. Applying the same conclusion to $-f$ we get that the upper gradient property is true for $\ppi$-almost every curve. Since $\ppi$ was an arbitrary  $q$-plan, by definition we have
$$ |f(\gamma_0) - f(\gamma_1)| \leq  \int_{\gamma} | \nabla f |_w \, \d s \qquad \text{ for $p$-almost every curve } \gamma$$
and so $|\nabla f|_w$ is a $p$-weak upper gradient.
\end{proof}

\newpage

\section{$BV$ space via derivations}

From now on, when $\mu \in \mathcal{M}(X)$, we will denote $\int_X \, \d \mu = \mu(X)$.

\begin{definition}\label{def:deri}  Let $f \in L^1(X, \sfd,\mm)$; we say $f \in BV(X, \sfd, \mm)$ if there exists a linear map $L_f: \Derri \to \mathcal{M}(X)$ satisfying
\begin{equation}\label{eq:defw}
\int_X \, \d L_f(\bb) =-\int_X f \div \bb\,d\mm\qquad \forall \text{ $\bb\in \Deri$,}
\end{equation}
continuous with respect to the $\Derr \infty$ norm and such that $L_{f}(h \bb)=hL_f(\bb)$ for every $h \in C_b(X), \bb \in \Derri$.
\end{definition}

As in the $W^{1,p}$ case, we can prove that $L_f(\bb)$ is uniquely defined whenever $f \in BV$ and $\bb \in \Deri$:

\begin{remark}[Well posedness in $\Deri$] Let us fix $\bb \in \Deri$, $f \in BV$; let $L_f$ and $\tilde{L}_f$ be two different linear maps given in the definition on $BV$. Let $h \in \Lipb$: using Lemma \ref{lem:multi} we have $ h \bb \in \Deri$ and so we can use the $C_b$ linearity and \eqref{eq:defw} to get
$$ \int_X h \, \d L_f(\bb)= \int_X \, \d L_f (h \bb)= -\int_X f \div ( h\bb) \, \d \mm,  $$
and the same is true for $\tilde{L}_f$. In particular $\int_X h \, \d L_f( \bb) = \int_X h \, \d \tilde{L}_f (\bb)$, and thanks to the arbitrariness of $h \in \Lipb $ we conclude that $L_f( \bb) = \tilde{L}_f( \bb)$. We will call this common value $Df (\bb)$.
\end{remark}

Now we can give the definition of total variation:

\begin{theorem}\label{thm:improved1} Let $f \in BV(X, \sfd, \mm)$; then there exists a finite measure $\nu \in \mathcal{M}_+(X)$ such that, for every Borel set $A \subseteq X$,
\begin{equation}\label{eqn:thesis1} \int_A \, \d Df(\bb) \leq \int_A | \bb|^* \, \d \nu  \qquad \forall \bb \in \Deri,
\end{equation}
where $g^*$ denotes the upper semicontinuous envelope of $g$. The least measure that realizes this inequality is denoted with $| Df |$, the weak total variation of $f$. Moreover
\begin{equation}\label{eqn:total}
|Df|(X)= \sup \{ |Df(\bb)(X)| \; : \; |\bb| \leq 1, \bb \in \Deri \}.
\end{equation}
\end{theorem}

\begin{proof} We argue similarly to Theorem \ref{thm:improved}: by hypothesis we have that $f \in BV$ and so there exists a $C_b$-linear map $L_f: \Deri \to \mathcal{M}(X)$ such that $L_f(\bb) (X)  \leq C \| \bb \|_{L^{\infty}}$, where we can take $C=\sup \{ |Df(\bb)(X)| \; : \; |\bb| \leq 1, \bb \in \Deri \}$. Note that if $|\bb| \leq h$ where $h \in C_b(X)$ then we have that
\begin{equation}\label{eqn:compact}\int_K \, \d L_f (\bb) \leq C \sup_{x \in K} h(x) \qquad \forall K \subseteq X \text{ compact;}
\end{equation}
in fact, denoting with $\rho_{n} = \min \{ 1- n\sfd(x,K) \}$, we have that $\rho_n \to \chi_K$ pointwise and $0 \leq \rho_n \leq 1$ so, by dominated convergence theorem, 
$$ \int_K \, \d L_f(\bb) = \lim_{n \to \infty} \int_X \, \rho_n \, \d L_f(\bb) \leq C \lim_{n \to \infty} \| \rho_n \bb\|_{\infty} \leq C\lim_{n} \sup_{x \in X}  \rho_n (x) h (x) = C \sup_{x \in K} h(x),$$
where the last equality holds thanks to the compactness of $K$. Now, for every compact set $K \subseteq X$ and consider two functionals in the Banach space $Y=C_b(K)$:
\begin{equation}\label{eqn:psi11}
 \Psi_2 ( h) = C \| h \|_{\infty}
\end{equation}
\begin{equation}\label{eqn:psi21}
 \Psi_1 ( h) = \sup \left\{ \int_K \, \d L_f(\bb)  : \bb \in \Derri,  \; \text{ $ \exists \, \tilde{h} \in C_b(X)$ such that $| \bb| \leq \tilde{h}$, $\tilde{h}|_{K} \leq h$} \right\}
\end{equation}
where the supremum of the empty set is meant to be $- \infty$. Equation \eqref{eqn:compact} guarantees that
\begin{equation}\label{eqn:ovvio21}
 \Psi_1( h) \leq \Psi_2 (h) \qquad \forall h \in Y.
\end{equation}
Moreover, as before, $\Psi_2$ is convex and continuous while $\Psi_2$ is concave; by Hahn-Banach theorem we can find a continuous linear functional $L$ on $C_b(K)$ such that
$$ \Psi_1(h) \leq L(h) \leq \Psi_2(h). $$
In particular there exists a measure $\mu_K$ such that $L(h) = \int_{K} h \, \d \mu_K$ and, thanks to \eqref{eqn:psi11}, we have $\mu_K(K) \leq C$. Moreover, thanks to \eqref{eqn:psi21} we have that if $h \in C_b(X)$ is such that $|\bb|\leq h$ for some $\bb \in \Derri$ then
$$ \int_K \, \d L_f( \bb) \leq \int_K  h \, \d \mu_K;$$
since for every $k \in C_b(X)$, we have $|k\bb| \leq |k|h$ we obtain also
$$\int_K k \, \d L_f(\bb) \leq \int_K |k| h \, \d \mu_K.$$
In particular, optimizing in $k$ we obtain also that $|L_f(\bb)|$, the total variation of $L_f(\bb)$, restricted to $K$, is less then or equal to $h \mu_K$. This implies that the following set is nonempty:
$$ A_K =  \left\{ \nu \in \mathcal{M}_+(K) \; : \; |L_f(\bb)||_K \leq h \nu  \text{ whenever }\bb \in \Derri , h \in C_b(X) \text{ s.t. } | \bb| \leq h \right\}. $$
Clearly this set is convex, weakly-$*$ closed and a lattice, in particular there exists the minimum, that we call $\nu_K$. We can drop the dependence on $K$ since it is easy to see that if $A \subset K_1 \cap K_2$ then $\nu_{K_1} (A) =\nu_{K_2} (A)$; suppose on the contrary that $\nu_{K_1}(A) > \nu_{K_2}(A)$. Then we can consider the measure $\tilde{\nu} (B)= \nu_{K_1}(B \setminus A)  + \nu_{K_2} ( B \cap A)$ that would be a strictly better competitor than $\mu_{K_1}$ in $A_{K_1}$.

Now we can extend $\nu$ to a measure on the whole space
$$\nu (B) = \sup_{ K \subseteq B} \nu (K)  \qquad \forall B \subseteq X \text{ Borel}; $$
this is easily seen to be a measure, that is also finite since $\nu(K) \leq \mu_K(K) \leq C$ for all $K$ compact and in particular we get $\nu(X) \leq C$. Thanks to the finiteness of $|L_f(\bb)|$ and $\nu$, using that $\nu|_K \in A_K$, we find that
$$ |L_f( \bb)| \leq h \nu \text{ whenever }\bb \in \Derri , h \in C_b(X) \text{ s.t. } | \bb| \leq h,$$
in particular, integrating in $A$ we get 
$$\int_A \, \d L_f(\bb) \leq \int_A h \, \d \nu, $$
and taking the infimum in $h$ we obtain \eqref{eqn:thesis1}, recalling that if $g \in L^{\infty}$ then
$$ g^*(x)= \inf \{ h(x) \; : \; h \in C_b(X), \; \; h \geq g  \; \mm \text{-a.e.}\}. $$
For the last assertion we already proved $C \geq \nu(X)$, while the other inequality is trivial taking $A=X$ in \eqref{eqn:thesis1}.
\end{proof}

\begin{theorem}[Representation formula for $|Df|$] Let $f \in BV$. Then the classical representation formula holds true: for every open set $A$
\begin{equation}\label{eqn:reform}
|Df|(A) = \sup \left\{ \int_A f \cdot \div (\bb) \, \d \mm \; : \; \bb \in \Deri , \; |\bb| \leq 1, \; {\rm supp} (\bb) \Subset A \right\},
\end{equation}
where $B \Subset A$ if $\sfd ( X \setminus A, B ) >0$ and $B$ is bounded.
\end{theorem}

\begin{proof} Let us consider two open sets $A_1, A_2$ and a closed set $C$ such that $A_1 \Subset C \Subset A_2$. We will consider $(C, \sfd, \mm)$ as a separable metric measure space, and relate the definitions of bounded variation in $X$ and $C$. Let us consider a function $f \in BV(X, \sfd,\mm)$; it is clear that $f \in BV(C, \sfd,\mm)$ since $\Deri (C) \subset \Deri (X)$ (it is sufficient to set $\bb_X(f)=\bb_C( f|_C)$), and consequently $|Df|_X \geq |Df|_C$ by \eqref{eqn:thesis1}.

Moreover $|Df|_X(A_1)=|Df|_C(A_1)$. This is true because there exists a Lipschitz function $0 \leq \chi \leq 1$ such that $\chi=0$ in $X \setminus C$ and $\chi=1$ on a neighborhood of $A_1$; then we have that if $\bb \in \Deri(X)$ implies that $\chi \bb \in \Deri(C)$ and so in \eqref{eqn:thesis1} we can imagine that $ \bb \in \Deri(C)$ whenever $A \subseteq A_1$; but then we get that the measure $\nu$ defined as
$$\nu(B)=|Df|_X(B \setminus A_1)+|Df|_C (B \cap A_1)$$
is a good candidate in \eqref{eqn:thesis1} and so, by the minimality  of $|Df|_X$ we get $|Df|_C(A_1)=|Df|_X(A_1)$.

Now, denoting by $\mu(A)$ the set function defined in the left hand side of \eqref{eqn:reform}, it is obvious that $\mu(A_2) \leq |Df|(A_2)$. But it is also obvious that $\mu (A_2) \geq |Df|_C(C) \geq |Df|_C(A_1)=|Df|_X(A_1)$. Letting $A_1 \up A_2$ we get the desired inequality.
\end{proof}

\subsection{Equivalence of $BV$ spaces}\label{sec:eq1}

We just sketch the equivalence with the other definitions given in literature: in particular we refer to \cite{ADM}, where the authors consider the spaces $BV_*$ and $w-BV$ and show their equivalence. As we did for $W^{1,p}$ we show $BV_* \subseteq BV \subseteq w-BV$.

\begin{lemma}\label{lem:1eq} Let $f \in BV_*(X, \sfd, \mm)$. Then we have $f \in BV(X, \sfd, \mm)$ and $|Df| \leq |Df|_*$ as measures.
\end{lemma}
\begin{proof}
By hypothesis, we know that there is a sequence $(f_n) \subset \Lipk$ such that ${ \rm lip}_a (f_n) \rightharpoonup |Df|_* $ in duality with $C_b(X)$; in particular, for every $\bb \in \Deri$ we have
$$ \left| \int_X f_n \cdot \div \bb \, \d \mm \right| = \left| \int_X \bb ( f_n) \, \d \mm \right | \leq \int_X | \bb| \cdot {\rm lip}_a (f_n) \, \d \mm$$
taking limits and recalling that whenever $\nu_n \rightharpoonup \nu$ and $g \geq 0$, we have $\liminf_{n \to \infty} \int_X g \, \d \mu_n \leq \int_X g^* \, \d \mu$, we have that
$$\left| \int_X f \cdot \div \bb \right | \leq \int_X | \bb |^* \, \d |Df|_* \qquad \forall \bb \in \Deri.$$
Now this inequality would guarantee that $|Df| \leq |Df|_*$ once we construct the linear functional $L_f: \Deri \to \mathcal{M}(X)$ %and then extend it to $\Derri$ via Lemma \ref{lem:hbc}.
In order to find $L_f(\bb)$ we proceed exactly as in Subsection \ref{ss:1eq}, and so we omit the construction.
\end{proof}

\begin{lemma}\label{lem:2eq} Let $f \in BV(X, \sfd, \mm)$. Then we have $f \in w-BV(X, \sfd, \mm)$ and $|Df|_w(X) \leq |Df|(X)$.
\end{lemma}
\begin{proof}
As for the second inclusion it is sufficient to recall Proposition \ref{prop:pipi}: we know that for every $\infty$-plan $\ppi$ we can associate a derivation $\bb_{\sppi} \in \Deri$; we use this derivation in the definition of $BV$ and, using also Theorem \ref{thm:improved1}, we obtain
$$- \int_X f \cdot \div \bb_{\sppi} \, \d \mm  \leq \int_X |\bb_{\sppi}|^* \, \d |Df|.$$
Now, using \eqref{def:modul} and \eqref{def:bspi2}, we obtain 
\begin{equation}
\label{eqn:almost1}
\int_{AC} (f(\gamma_0) - f(\gamma_1)) \d \ppi \leq C(\ppi) \cdot |Df|(X) \| {\rm Lip} (\gamma) \|_{L^{\infty}(\sppi)} \qquad \forall \ppi \text{ $\infty$-plan. }
\end{equation}
Now we can use Remark 7.2 in \cite{ADM} to conclude that $f \in w-BV$ and $|Df|_w(X) \leq |Df|(X)$
\end{proof}

Using this two lemmas in conjunction with the equivalence result in \cite{ADM} we can conlcude.

\begin{theorem}\label{th:finalbv} Let $(X,\sfd,\mm)$ be a complete and separable metric space, such that $\mm$ is finite on bounded sets; then $BV(X, \sfd,\mm)=BV_*(X,\sfd,\mm)=w-BV(X,\sfd,\mm)$. Moreover $|Df|=|Df|_*=|Df|_w$ for every function $f \in BV$.
\end{theorem}

\begin{proof} From Lemma \ref{lem:1eq} and \ref{lem:2eq} we know that $BV_* \subseteq BV \subseteq w-BV$ and moreover $|Df| \leq | Df |_* $ and $|Df|(X) \geq |Df|_w(X)$. Thanks to the equivalence theorem in \cite{ADM} we get $BV=BV_*=w-BV$ and $|Df|_w=|Df|_*$, in particular $|Df|_w(X)=|Df|_*(X) \geq |Df|(X)$, and so $|Df|(X)=|Df|_*(X)=|Df|_w(X)$. This equality, along with $|Df| \leq | Df |_* $ let us conclude that the three definitions of total variation coincide.
\end{proof}

\def\cprime{$'$} \def\cprime{$'$}


\begin{thebibliography}{10}
%
%\bibitem{Shanmugalingham10}
%{\sc T.~Adamowicz and N.~Shanmugalingam}, {\em Non-conformal Loewner type estimates for modulus
%of curve families,} Annales Academiae Scientiarum Fennicae, {\bf 35} (2010), 609--626.

%\bibitem{Ambrosio-Tilli}
%{\sc L.~Ambrosio and P.~Tilli}, {\em Selected topics on Analysis in Metric Spaces}, Oxford University Press,
%2004.

\bibitem{ADM}
{\sc L.~Ambrosio and S.~Di Marino}, {\em Equivalent definitions of BV space and of total variation on metric measure spaces}, J. Funct. Anal., {\bf 266} 7 (2014), 4150--4188.
Preprint (2012).

\bibitem{ADMS}
{\sc L.~Ambrosio, S.~Di Marino, and G.~Savar{\'e}}, {\em On the duality between {$p$}-modulus and probability measures}, preprint, 2013.

%\bibitem{AFP}
%{\sc L.~Ambrosio, N.~Fusco, and D.~Pallara}, {\em Functions of bounded variation and free discontinuity problems.}
%Oxford University Press, 2000.

%\bibitem{Ambrosio-Gigli-Savare08}
%{\sc L.~Ambrosio, N.~Gigli, and G.~Savar{\'e}}, {\em Gradient flows in metric
%  spaces and in the space of probability measures}, Lectures in Mathematics ETH
%  Z\"urich, Birkh\"auser Verlag, Basel, second~ed., 2008.
%
%\bibitem{Ambrosio-Gigli-Savare11}
%\leavevmode\vrule height 2pt depth -1.6pt width 23pt, {\em Calculus and heat
%  flows in metric measure spaces with {R}icci curvature bounded from below},
%  arXiv:1106.2090,  to appear on Inventiones Mathematicae.
%  
\bibitem{AGS11}
{\sc L.~Ambrosio, N.~Gigli, and G.~Savar{\'e}}, {\em Density of Lipschitz functions and 
equivalence of weak gradients in metric measure spaces}, Rev. Mat. Iberoam., {\bf 29} 3  (2013), 969--996

%\bibitem{Ambrosio-Gigli-Savare11bis}
%\leavevmode\vrule height 2pt depth -1.6pt width 23pt, {\em Metric measure spaces with Riemannian
%curvature bounded from below}, Arxiv 1109.0222,  (2011), pp.~1--60.

\bibitem{atrev}
{\sc L.~Ambrosio, D.~Trevisan}, {\em Well posedness of Lagrangian flows and continuity equations in metric measure spaces}
preprint, 2014.

\bibitem{Bate}
{\sc D.~Bate}, {\em Structure of measures in Lipschitz differentiability spaces,}
ArXiv preprint 1208.1954, (2012).

\bibitem{Bjorn-Bjorn11} 
{\sc A.~Bj{\"o}rn and J.~Bj{\"o}rn}, {\em Nonlinear potential theory
  on metric spaces.} EMS Tracts in Mathematics, 17, 2011.

\bibitem{Cheeger00}
{\sc J.~Cheeger}, {\em Differentiability of {L}ipschitz functions on metric
  measure spaces}, Geom. Funct. Anal., {\bf 9} (1999), 428--517.

% \bibitem{Fuglede}
%{\sc B.~Fuglede}, {\em Extremal length and functional completion}, Acta Math.,
%  {\bf 98} (1957), 171--219.
%  
%\bibitem{Gigli12}
%{\sc N.~Gigli}, {\em On the differential structure of metric measure spaces and
%  applications},  (2012).
  
  \bibitem{Gigli14}
{\sc N.~Gigli}, {\em Nonsmooth differential geometry : an approach tailored for spaces with Ricci curvature bounded from below}, Preprint,  (2014).

 \bibitem{Hajlasz-Koskela}
{\sc P.~Hajlasz and P.~Koskela}, {\em Sobolev met Poincar\'e}, Mem. Amer. Math. Soc., {\bf 154} (2000) 688, 
 
\bibitem{Heinonen07}
{\sc J.~Heinonen}, {\em Nonsmooth calculus}, Bull. Amer. Math. Soc., {\bf 44} (2007),
 163--232.
%
%\bibitem{Heinonen-Koskela98}
%{\sc J.~Heinonen and P.~Koskela}, {\em Quasiconformal maps in metric spaces
%  with controlled geometry}, Acta Math., {\bf 181} (1998), 1--61.
%
%\bibitem{Heinonen-Koskela99}
%\leavevmode\vrule height 2pt depth -1.6pt width 23pt, {\em A note on
%  {L}ipschitz functions, upper gradients, and the {P}oincar\'e inequality}, New
%  Zealand J. Math., {\bf 28} (1999), 37--42.
% 
  
\bibitem{Schioppa}
{\sc A. Schioppa}, {\em Derivations and Alberti representations},
preprint, (2013).

\bibitem{Schioppacur}
\leavevmode\vrule height 2pt depth -1.6pt width 23pt, {\em Metric Currents and Alberti representations},
preprint, (2014).
  
%  
%\bibitem{Shanmugalingham01}
%{\sc S.~Kallunki and N.~Shanmugalingam}, {\em Modulus and continuous capacity},
%Annales Academiae Scientiarum Fennicae, {\bf 26} (2001), 455--464.
  
\bibitem{Koskela-MacManus}
{\sc P.~Koskela and P.~MacManus}, {\em Quasiconformal mappings and {S}obolev
 spaces}, Studia Math., {\bf 131} (1998), 1--17.

\bibitem{Shanmugalingam00}
{\sc N.~Shanmugalingam}, {\em Newtonian spaces: an extension of {S}obolev
  spaces to metric measure spaces}, Rev. Mat. Iberoam., {\bf 16} (2000),
  243--279.

\bibitem{W00}  
{\sc N.~Weaver}, {\em Lipschitz algebras and derivations II. Exterior differentiation},
J. Funct. Anal., {\bf 178} 1 (2000), 64--112.


\end{thebibliography}
\end{document}